\documentclass[11pt]{amsart}
\usepackage{amsmath}
\usepackage{amsfonts}
\usepackage{amssymb}
\usepackage{amscd}
\usepackage{amsthm}
\usepackage{amsbsy}
\usepackage{graphicx}
\usepackage{bm}
\usepackage{enumerate}
\usepackage[capitalize]{cleveref}
\newtheorem{theorem}{Theorem}[section]
\newtheorem{proposition}[theorem]{Proposition}
\newtheorem{lemma}[theorem]{Lemma}
\newtheorem{corollary}[theorem]{Corollary}

\theoremstyle{definition}
\newtheorem{definition}[theorem]{Definition}
\newtheorem{remark}[theorem]{Remark}

\newcommand{\hd}{d_{\vec{H}}}
\newcommand{\gl}{\mathcal{GL}}
\newcommand{\al}{\mathcal{AL}}
\newcommand{\g}[1]{\operatorname{genus}(#1)}

\title[Bijections of geodesic lamination space]{Bijections of geodesic lamination space preserving left Hausdorff convergence}
\author[Ken'ichi Ohshika and Athanase Papadopoulos]{Ken'ichi Ohshika and Athanase Papadopoulos}

 \address{A. Papadopoulos, Institut de Recherche Math\'ematique Avanc\'ee (Universit\'e de Strasbourg et CNRS),
7 rue Ren\'e Descartes
67084 Strasbourg Cedex France,  papadop@math.unistra.fr ; K. Ohshika, 
Department of Mathematics, Graduate School of Science, Osaka University Toyonaka, Osaka 560-0043, Japan, 
ohshika@math.sci.osaka-u.ac.jp}

\begin{document}
\maketitle
  \begin{abstract}
  We introduce an asymmetric distance function, which we call the ``left Hausdorff distance function", on the space of geodesic laminations on a closed hyperbolic surface of genus at least 2. This distance is an asymmetric version of the Hausdorff distance between compact subsets of a metric space. We prove a rigidity result for the action of the extended mapping class group of the surface on the space of geodesic laminations equipped with the topology induced from this distance.  More specifically, we prove that there is a natural homomorphism from the extended mapping class group into the group of bijections of the space of geodesic laminations that preserve left Hausdorff convergence and that this homomorphism  is an isomorphism.
  
  The final version of this paper will appear in \emph{Monatshefte f\"ur Mathematik}.

 \medskip
 

 \medskip
 
 
 \end{abstract}

\section{Introduction}
Let $S$ be a connected closed orientable surface of genus $g\geq 2$. Let $\mathrm{Mod}(S)$ be the mapping class group of $S$, that is, the group of homotopy classes of orientation-preserving homeomorphisms of $S$, and $\mathrm{Mod}^*(S)$ the extended mapping class group of $S$, that is, the group of all homotopy classes of homeomorphisms of $S$.
 In the last three decades, motivated by Thurston's works on surfaces, a recurrent theme in low-dimensional topology was the study of actions of the groups $\mathrm{Mod}(S)$ and $\mathrm{Mod}^*(S)$ on various spaces associated to the surface $S$. The spaces that were considered are equipped with different kinds of structures, including the following: 
 \begin{enumerate}
 \item complex analytic structures, including the complex structures of the universal Teichm\"uller curve \cite{T32}, of Teichm\"uller space \cite{Royden1971, Markovic2003}, of spaces of quadratic differentials with the $L^1$-norm  \cite{Royden1971, Markovic2003}, etc.;
 \item  metric structures: the various metrics on Teichm\"uller space, including the Teichm\"uller metric \cite{Royden1971, Markovic2003}, the Weil-Petersson metric \cite{MW2002} and the Thurston metric \cite{Walsh-Hand};
\item combinatorial structures: the curve complex \cite{Iv}, the arc complex \cite{IM, Disarlo}, the arc and curve complex \cite{KP1}, the complex of domains \cite{MCP}, and there are many others;
\item  the piecewise-linear structure of the space of measured laminations on the surface \cite{Papadopoulos2008};
\item topological structures: actions by homeomorphisms on spaces of laminations or foliations \cite{AMO, CPP, Pa, Oh, Oh1, OP}. 
\end{enumerate}
These are only examples; there are other classes of spaces equipped with actions of mapping class groups. 

The main result obtained in most of these cases is that the structure considered is rigid in the sense that its automorphism group coincides with the natural injective image of the (extended) mapping class group in it. (We are leaving aside some exceptional cases of surfaces---finite in number---of low genus and with a small number of boundary components which arise in each case.) There is an  exception though, namely, the case of the complex of domains, where there is a large class of automorphisms of the complex that are not induced by mapping classes, see \cite{MCP} where these automorphisms are called ``exchange automorphisms".

 Although the rigidity results we mentioned almost always have the same form, the methods and the techniques of proof in each case require new tools that are specific to the setting. At the same time, these methods highlight new properties of the spaces under consideration.
We refer the reader to the expository paper \cite{Pa1} for an overall view on this topic.

The space $\mathcal{GL}(S)$ of geodesic laminations (for a fixed hyperbolic metric) was introduced by Thurston in his lecture notes \cite{Th} (see Chapter 8, and in particular \S \ 8.1). This space plays a major role in Thurston's theory of 3-manifolds and Kleinian groups. It is classically equipped with two different topologies, also introduced by Thurston, the so-called Thurston topology (which he called the geometric topology) and the Hausdorff topology. The homeomorphism group of $\mathcal{GL}(S)$ with respect to the Thurston topology was studied in the paper  \cite{CPP}, where it was shown that this group coincides with the natural image in it of the extended mapping group. It was also noted in the same paper that the homeomorphism group of $\mathcal{GL}(S)$  with respect to the Hausdorff topology is much larger than the image of the extended mapping class group in it. The reason is that simple closed geodesics are isolated points in $\mathcal{GL}(S)$, therefore any two such points  can be permuted by homeomorphisms of the space; in particular we can send a separating curve to a non-separating curve by a homeomorphism of $\mathcal{GL}(S)$ and such a homeomorphism is obviously not induced by an element of the mapping class group.

In the present paper, we study the action of $\mathrm{Mod}^*(S)$ on the space of geodesic laminations  $\mathcal{GL}(S)$ equipped with a new structure which we now introduce. This is a topological structure on $\mathcal{GL}(S)$ induced by an asymmetric version of the Hausdorff distance. We prove a rigidity result concerning this structure. 
 
 Before stating precisely the result, we note that the study of asymmetric distance functions in Teichm\"uller theory has been an active research theme since the introduction by Thurston of his asymmetric metric on Teichm\"uller space. Among the early works on this subject, we mention the paper \cite{P1991} in which the author determines the limits of stretch rays (which are geodesics for this metric).
   The paper  \cite{2007a} by Papadopoulos and Th\'eret is concerned with the topology associated to this metric. In the paper \cite{Theret}, Th\'eret studies the convergence at infinity of anti-stretch lines (these are geodesics traversed in the backward direction).
In the paper \cite{2012b},   Liu, Papadopoulos, Su and Th\'eret obtain a classification of mapping classes by analysing their action on this metric.
  In the paper  \cite{2009j}, Papadopoulos and Th\'eret construct geodesic lines that remain geodesic (up to reparametrisation) when they are traversed in the backward direction.
   In the paper  \cite{Walsh-Hand}, Walsh proves a rigidity result for the action of the mapping class group on this metric,  and there are other papers on the subject. Furthermore, asymmetric metrics whose definition mimics the Thurston metric are now studied in various contexts,  see e.g. the paper \cite{GK} by Gu\'eritaud and Kassel for an asymmetric metric on a class of 3-manifolds, the paper \cite{AB} by Algom-Kfir and Bestvina for an asymmetric metric on outer automorphism groups of free groups, and the paper \cite{SM} by Meinert for an asymmetric metric on deformation spaces of $G$-trees. We also mention the papers  \cite{DGK} by  Danciger,  Gu\'{e}ritaud and Kassel,  where Thurston's asymmetric metric is used in relation with AdS geometry and \cite{GLM} by Goldman,  Labourie and Margulis in relation with deformation spaces of proper affine actions on $\mathbb{R}^3$. Let us also note that   geodesic laminations of various kinds appear in some contexts as   tangent vectors to Teichm\"uller space, and equipping the space of geodesic laminations with an asymmetric distance (as we do in the present paper) may be regarded in some sense as equipping the tangent space of Teichm\"uller space  with such a distance. In his foundational paper \emph{Minimal stretch maps between hyperbolic surfaces} \cite{Th1} (l.\,7 of p.\,40), Thurston makes an elliptical remark on such a distance function (or, rather, a topology associated to such a distance function), when he talks about ``a non-Hausdorff topology on the set of chain recurrent laminations, where a neighborhood of a lamination consists of all laminations contained in a neighborhood of the lamination of the surface''. This is precisely the property that defines the topology associated to the asymmetric distance function on $\mathcal{GL}(S)$ which we are considering in this paper.

Finally, we mention that in Finsler geometry, asymmetric distances play an important role. As a matter of fact, in this field, if a metric is symmetric, it is called \emph{reversible}. 

 Before introducing our asymmetric distance on the space of geodesic laminations, we start with a few definitions.

 We denote by $d_m$ the distance function on $S$ induced by a fixed hyperbolic metric $m$.
\begin{definition}
For any ordered pair of compact subsets $K, L\subset S$, the \emph{left Hausdorff distance} $\hd(K,L)$  from $K$ to $L$ is defined as
$$\hd(K,L)=\inf\{\epsilon \mid K \subset N_\epsilon(L)\},$$
where for $\epsilon \geq 0$, $N_\epsilon$ denotes the $\epsilon$-neighbourhood with respect to $d_m$.
\end{definition}
It is easy to see by examples that the left Hausdorff distance from $K$ to $L$ is generally different from the left Hausdorff distance from $L$ to $K$.

Being a space of compact subspaces of $S$, $\mathcal{GL}(S)$ is equipped with an induced left Hausdorff distance function which we also denote by  $\hd$.  

\begin{definition} 
 
 Let $f : \gl(S) \to \gl(S)$ be a bijection.
We say that $f$ preserves left Hausdorff convergence if for any sequence $\{\lambda_i \in \gl(S)\}$ and for any $\mu \in \gl(S)$, we have
$$\hd(\lambda_i, \mu) \to 0 \Leftrightarrow \hd(f(\lambda_i), f(\mu)) \to 0.$$
\end{definition}
It is easy to see that the following equivalence holds for any bijection $f$ preserving left Hausdorff convergence:
 \[\hd(\lambda, \mu)=0 \Leftrightarrow \hd(f(\lambda), f(\mu))=0.\]

We let  $\mathrm{Aut}( \gl(S))$ be the group of bijections of $ \gl(S)$ that preserve  
left Hausdorff convergence. We have a natural homomorphism \[\mathrm{Mod}^*(S)\to\mathrm{Aut}( \gl(S)).\]
The aim of this paper is to prove the following.

\begin{theorem}
\label{main}
The natural homomorphism \[\mathrm{Mod}^*(S)\to\mathrm{Aut}( \gl(S))\] is an isomorphism.
\end{theorem}
\sloppy

\begin{remark}
It is possible to define a topology associated to the left Hausdorff distance by taking the sets of the form $U_\epsilon(\lambda)=\{\mu \mid \hd(\mu, \lambda) < \epsilon\}$ as a basis for a fundamental system of neighbourhoods of a lamination $\lambda$. The result of this paper can then be formulated in terms of homeomorphisms of $\mathcal{GL}(S)$ with respect to this topology. 
\end{remark}
 \section{Preliminaries}\label{s:pre}
In this section, we recall a few definitions concerning geodesic laminations and related matters on hyperbolic surfaces. We shall use all these definitions later in the paper. 
We refer the reader to \cite{Th, CEG, BC, MT} for more details on this topic.

Let $S$ be, as before, a closed orientable surface equipped with a hyperbolic structure $m$.
A \emph{geodesic lamination} on $S$ is a closed subset of $S$ which is the union of disjoint simple geodesics, called the \emph{leaves} of the lamination.
A \emph{component} of a geodesic lamination $\lambda$ is a geodesic lamination on $S$ which is a subset of $\lambda$.
A lamination is said to be \emph{minimal}  if it contains no non-empty proper sublamination.
A leaf of a geodesic lamination is said to be \emph{isolated} if, as a subset of $S$, it has a neighbourhood containing no other leaf than itself.
Any geodesic lamination $\lambda$ has a unique decomposition into finitely many minimal components and finitely many non-compact isolated leaves. Note that in this statement, the non-compact isolated leaves are not minimal components, the reason being that they do not constitute a sublamination (they are not closed subsets of $S$). 
An isolated leaf is either a closed geodesic or non-compact. We shall use the fact that each end of a non-compact isolated leaf spirals around some minimal component of $\lambda$, and we refer the reader to \cite[Theorem 1.4.2.8]{CEG} for a proof of this fact.

In this paper, laminations consisting only of isolated leaves play an important role.

\begin{definition}
\label{finite lamination}
A geodesic lamination is called {\em finite} if all its minimal components are simple closed geodesics.
\end{definition}

(Note that a finite lamination may contain components which are not simple closed geodesics.)

The following result is due to Thurston \cite{Th}. A detailed proof can be found in \cite[Theorem 4.2.14]{CEG}.

\begin{lemma}
\label{finite laminations dense}
The set of finite laminations is dense in $\gl(S)$ with respect to the Hausdorff topology.
\end{lemma}

\section{Actions on curves} 
From this section until \cref{s:finite}, we assume that $f$ is a bijection of $\mathcal{GL}(S)$ preserving left Hausdorff convergence.

\begin{lemma}
For any simple closed geodesic $c$, its image $f(c)$ is again a simple closed geodesic.
\end{lemma}

This is a consequence of the following characterisation of simple closed geodesics in terms of $\hd$.

\begin{lemma}
\label{scc}
Let $c$ be a simple closed geodesic, and suppose that $\hd(\lambda_i, c) \to 0$ for some sequence $\{\lambda_i\} \subset \gl(S)$.
Then $\lambda_i=c$ for large $i$.

Conversely, if $\hd(\lambda_i, \mu) \to 0$ implies that $\lambda_i=\mu$ for large $i$, then $\mu$ is a simple closed geodesic.
\end{lemma}
\begin{proof}
The first half of the statement is  easy. Indeed, for any sufficiently small $\epsilon >0$, the $\epsilon$-neighbourhood of $c$ is an annulus, and we can see that any geodesic lamination contained in such a neighbourhood must be equal to $c$.

For the second half, let $\mu$ be a geodesic lamination satisfying the condition in the statement.
Let $\mu_0$ be a minimal component of $\mu$ (see \S \ref{s:pre}).
Then we have $\hd(\mu_0, \mu)=0$.
Therefore by the assumption of the second half of the lemma, $\mu_0=\mu$, which means that $\mu$ consists of only one minimal component and does not have any non-compact isolated leaf.
If $\mu$ is not a simple geodesic, then the minimal component can be approximated by a sequence of simple closed geodesics $c_i$ in the Hausdorff topology as follows.
Take a leaf $l$ of $\mu$. For each positive integer $i$, choose an arc $a_i$ on $l$ with length greater than $i$ whose endpoints can be joined by a geodesic arc $b_i$ transverse to $l$ of length less than $1/i$ and such that the endpoints of the arc $a_i$ arrive on different sides of $b_i$.
Since $l$ is dense in $\mu$, the closed geodesic $c_i$ homotopic to $a_i\cup b_i$ converges to $\lambda$ in the Hausdorff topology.
For such a sequence $\{c_i\}$, we have $\hd(c_i, \mu) \to 0$,  contradicting the assumption made.
Thus the only possibility is that $\mu$ is a simple closed geodesic.
\end{proof}

We next show that the inclusion relation between geodesic laminations is preserved by $f$.

\begin{lemma}
\label{inclusion}
Suppose that $\lambda \subset \mu$ for two geodesic laminations.
Then $f(\lambda) \subset f(\mu)$.
\end{lemma}
\begin{proof}
This follows from the equivalence $\lambda \subset \mu \Leftrightarrow \hd(\lambda, \mu)=0$.
\end{proof}

A geodesic lamination consisting of a collection of disjoint simple closed geodesics on $S$ whose number of connected components is $\geq 1$ will be called a {\em multicurve}. 
(Thus, we regard a simple closed geodesic also as a multicurve.) From now on, let us abuse notation and write ``component" instead of ``connected component" for multicurves (this notion of ``component" does not coincide with our definition of 
 component of a lamination).

We can characterise multicurves with at least two components as follows:

\begin{lemma}
\label{multicurve}
A geodesic lamination $\mu$ is a multicurve, but not a simple closed geodesic  if and only if the following conditions hold:
\begin{enumerate}[(a)]
\item If $\hd(\lambda_i, \mu) \to 0$, then $\lambda_i$ is contained in $\mu$ for large $i$.
\item $\mu$ is the union of geodesic laminations properly contained in $\mu$.
\end{enumerate}
\end{lemma}

\begin{remark} The second condition is necessary since the first condition alone is satisfied by a union of a simple closed geodesic and one single non-compact isolated leaf spiralling around it on one side.
\end{remark}

\begin{proof}[Proof of Lemma \ref{multicurve}]
It is clear that multicurves which are not simple closed geodesics satisfy these two conditions.
To prove the converse, suppose that $\mu$ satisfies these two conditions. Then, as was shown in the proof of \cref{scc}, $\mu$ cannot have any minimal component which is not a simple closed geodesic.
Condition (b) implies that $\mu$ is the union of its minimal components, and that there are more than one components.
Therefore $\mu$ is a multicurve which is not a simple closed geodesic.
\end{proof}

By \cref{scc,multicurve}, $f$ takes any multicurve to a multicurve.

Now we show that $f$ preserves the number of components for multicurves.

\begin{lemma}
Let $n$ be a positive integer. If $\mu$ is a multicurve with $n$ components, then so is $f(\mu)$.
\end{lemma}
\begin{proof}
For $n=1$, the statement is nothing but \cref{scc}.
A multicurve $\mu$ with two components is characterised by the property that
``if $\mu$ contains $\lambda$ then either $\lambda$ is a simple closed geodesic or $\mu=\lambda$".
Therefore $f$ preserves this property.
Inductively, a multicurve $\mu$ with $n$ components is characterised by the property that
``if $\mu$ contains $\lambda$, then $\lambda$ is a multicurve with at most $n-1$ components or $\lambda=\mu$."
This property is also preserved by $f$.
\end{proof}

We also note that
$n$ simple closed geodesics are pairwise disjoint if and only if there exists a multicurve containing all of them.
Therefore disjointness of simple closed curves is also preserved by $f$.
Combining these properties, we see that $f$ induces an automorphism on the curve complex $\mathcal{C}(S)$ of $S$.
By Ivanov's theorem \cite{Iv}, this implies that there is a homeomorphism of $S$ inducing the same map as $f$ on $\mathcal{C}(S)$.
Thus we have the following.

\begin{corollary}
\label{Ivanov}
Let $f$ be a bijection on $\gl(S)$ preserving left Hausdorff convergence.
Then there is a homeomorphism $g: S \to S$ such that $f$ and $g$ induce the same simplicial automorphism on $\mathcal{C}(S)$.
\end{corollary}

Ivanov's theorem also shows that this homeomorphism $g$ is unique up to isotopy provided that $\g{S} \geq 3$.
When $\g{S}=2$, there are two choices of isotopy classes whose difference is represented by a hyperelliptic involution.
Indeed, the hyperelliptic involution $\iota$ acts on $\mathcal{C}(S)$ trivially, and hence $g$ and $\iota \circ g$ induce the same action on $\mathcal{C}(S)$.

\section{Approximable laminations}
In this section, $f$ is as before a bijection of $\gl(S)$ preserving left Hausdorff convergence, and $g$ denotes an automorphism of $S$ inducing the same map as $f$ on the curve complex $\mathcal{C}(S)$.
For any geodesic $\ell$ on $S$, we abuse the symbol $g(\ell)$ to denote the geodesic homotopic to $g(\ell)$.
In this way, we regard $g$ as acting on $\gl(S)$.

\begin{definition}
We say that a geodesic lamination $\mu$ is  {\em approximable} when there is a sequence of multicurves  $c_i$ which converges to $\mu$ in the (ordinary) Hausdorff topology.
We denote by $\al(S)$ the subset of approximable laminations of $\gl(S)$.
\end{definition}

%

\begin{lemma}
\label{multicurves}
If $\lambda$ is a union of its minimal components, then it is approximable.
\end{lemma}
\begin{proof}
First suppose that $\lambda$ is minimal.
Then as was shown in the proof of \cref{scc} there is a sequence of closed geodesics $\{c_i\}$ converging to $\lambda$ in the Hausdorff topology.
In the general case, we can take a sequence of closed geodesics $\{c_i^j\}$ as above for each minimal component $\lambda_j$ so that the $c_i^j \cap c_i^{j'}=\emptyset$ if $j \neq j'$.
Then the union $\cup_j c_i^j$ converges to $\lambda$ as $i \to \infty$ in the Hausdorff topology. 
\end{proof}


Since inclusion is preserved by $f$ (\cref{inclusion}), any minimal lamination is mapped by $f$ to a minimal lamination.

\begin{lemma}
\label{approx}
If $\mu$ is an approximable lamination, then there is a sequence of multicurves $\{c_i\}$ with the following two properties: \begin{enumerate}[(i)]
\item $\hd(c_i, \mu) \to 0$.
\item Any $\lambda$ such that $\hd(c_i, \lambda) \to 0$ contains $\mu$.
\end{enumerate} 
\end{lemma}
\begin{proof}
Let $\{c_i\}$ be a sequence of multicurves converging to $\mu$ in the Hausdorff topology.
Then by the definition of the Hausdorff topology and $\hd$, we have (i) and (ii).
\end{proof}

\begin{corollary}
\label{approximable rigidity}
If $\mu$ is an approximable lamination, then $f(\mu)=g(\mu)$.
\end{corollary}
\begin{proof}
Take $\{c_i\}$ as in \cref{approx}.
Then $\hd(f(c_i), f(\mu)) \to 0$, since $f$ preserves left Hausdorff convergence.
By \cref{inclusion}, if $\hd(f(c_i), \lambda)\to 0$, then $\lambda$ contains $f(\mu)$.
Since $g(c_i)=f(c_i)$, we have $\hd(f(c_i), g(\mu)) \to 0$, and hence $g(\mu)$ contains $f(\mu)$.
Since $g$ is a homeomorphism of $S$, it also preserves left Hausdorff convergence and inclusion.
Thus, by exchanging the roles of $f$ and $g$, $f(\mu)$ contains $g(\mu)$.
\end{proof}

\section{Non-compact isolated leaves}\label{s:non-compact}

\begin{definition}
\label{limit component}
Let $l$ be a non-compact isolated leaf of a geodesic lamination $\lambda$. Recall that each end of $l$ spirals around some minimal component of $\lambda$ (see \S \ref{s:pre}).
Thus, there are one or two minimal components around which the two ends of $l$ spiral. Fixing an orientation on $l$ and letting $L^+(l)$ be the limit component in the positive direction and $L^-(l)$ the one in the negative direction, we call these two minimal components the {\em limit components} of $l$ and denote them by  $L^+(l), L^-(l)$. 
\end{definition}
Note that the two limit components $L^+(l), L^-(l)$ may coincide.  
Note also that the distinction between the positive and the negative direction, and hence the orientation given on $l$, will turn out to be irrelevant in our argument as we shall see below.

The limit components are minimal components, and hence are contained in $\al(S)$.

\begin{lemma}
\label{smallest lamination}
Let $f$ be a bijection of $\gl(S)$ preserving left Hausdorff convergence.
Let $\lambda$ be a geodesic lamination and suppose that  $\lambda$ has a non-compact isolated  leaf $\ell$.
Let  $L^+(\ell), L^-(\ell)$ (possibly equal) be the limit components of $\ell$.
Then $f(\lambda)$ contains  $f(L^+(\ell)), f(L^-(\ell))$ as minimal components and  an isolated leaf having $f(L^+(\ell)), f(L^-(\ell))$ as its limit components.
\end{lemma}
\begin{proof}
Suppose that $\ell$ has two distinct limit components $L^+(\ell), L^-(\ell)$, and consider the geodesic lamination $L^+(\ell) \cup L^-(\ell) \cup \ell$, which we denote by $L$.
Then $L$ is a sublamination of $\lambda$.
By \cref{inclusion}, $f(\lambda)$ contains $f(L)$.
On the other hand, since $f$ is induced by a homeomorphism $g$ of $S$ on $\al(S)$, $f(L^+(\ell) \cup L^-(\ell))=g(L^+(\ell)) \cup g(L^-(\ell))$ is the union of two minimal laminations, which we denote by $L^+_f, L^-_f$.
Since $L^+(\ell) \cup L^-(\ell)$ is the union of all minimal components of $L$ and since this property is preserved by $f$,
 $L^+_f, L^-_f$ are the minimal components of $f(L)$.

By our definition,  $L$ contains $L^+(\ell) \cup L^-(\ell)$ properly and it is minimal among all laminations containing $L^+(\ell) \cup L^-(\ell)$ properly.
This property is preserved by $f$.
Therefore $f(L) \setminus (L^+_f \cup L^-_f)$ contains only one leaf, and it is a non-compact isolated leaf, which we denote by $\ell_f$.
If $\ell_f$ has only one of $L^+_f, L^-_f$, say $L^+_f$, as its limit component, then we have proper inclusions $f(L^+_f) \subsetneq f(L^+_f) \cup \ell_f \subsetneq f(L)$.
By applying $f^{-1}$ to these inclusions, we get a geodesic lamination $L'$ such that $L^+_f \subsetneq L' \subsetneq L$.  By our definition of $\ell$, this implies that $L'=L^+(\ell) \cup L^-(\ell)$.
This is a contradiction since we would have then $L^+_f\cup L^-_f=f(L')=f(L^+_f) \cup \ell_f$.

Thus, we have shown that $\ell_f$ has both $L^+_f$ and $L^-_f$ as limit components.
Since $\ell_f$ is contained in $f(L) \subset f(\lambda)$, we are done in this case.
The same kind of argument works also in the case when $\ell$ has only one limit component.
\end{proof}

\section{Finite laminations}\label{s:finite}

We next refine \cref{smallest lamination} to the case when $\lambda$ is a finite lamination to show that $f(\lambda)$ contains a leaf ``homotopic relative to compact leaves" to $g(\ell)$  (or $\iota \circ g(\ell)$ when $\mathrm{genus}(S)=2$  where $\iota$ is as before the hyperelliptic involution in genus $2$) as a unique non-compact isolated leaf. Here, we say that two non-compact leaves are homotopic relative to compact leaves if they spiral around the same pair of compact leaves and if they are homotopic to each other outside thin annular neighbourhoods of their limit components.
The formal definition is as follows:

\begin{definition}
Let $\ell$ be a non-compact isolated leaf of a finite lamination $\lambda \in \gl(S)$ with limit components $L^+, L^-$ ($L^+$ and $L^-$ may coincide).
Then the {\em homotopy class of $\ell$ relative to compact leaves} is defined to be the homotopy class of $\ell \setminus (A(L^+) \cup A(L^-))$ on $S \setminus (A(L^+) \cup A(L^-))$ where $A(L^+), A(L^-)$ denote annular neighbourhoods of $L^+$ and $L^-$ which are disjoint from each other and from the other minimal components of $\lambda$.
\end{definition}

\begin{lemma}
\label{homotopy class}
Let $\ell$ be a non-compact isolated leaf of a finite lamination $\lambda \in \gl(S)$.
Then there is a leaf of  $f(\lambda)$ which has the same limit components as $g(\ell)$.
The leaf is homotopic relative to compact leaves to $g(\ell)$ when $\g S \geq 3$.
When $\g S =2$ the leaf is homotopic relative to compact leaves to either $g(\ell)$ or $\iota \circ g(\ell)$, where $\iota$ is a hyperelliptic involution.
\end{lemma}
\begin{proof}
Construct a pants decomposition by taking disjoint simple closed geodesics in $S \setminus (L^+ \cup L^-\cup l)$, so that $L^+\cup L^-$ is contained in only one pair of pants if $\g S \geq 3$,  and denote it by $C$.
We have $f(C\cup L^+ \cup L^-)=g(C\cup L^+ \cup L^-)$ since $f$ and $g$ coincide on $\mathcal C(S)$.
Since  $L^+\cup L^- \cup \ell$ is a geodesic lamination contained in both  $\lambda$ and $C\cup L^+\cup L^- \cup \ell$, $f(L^+\cup L^- \cup \ell)$ is also a geodesic lamination contained in both $f(\lambda)$ and $f(C\cup L^+\cup L^- \cup \ell)$, which implies that $f(\lambda)$ contains a non-compact isolated leaf $l'$ disjoint from $g(C\cup L^+\cup L^-)$ with limit components $g(L^+), g(L^-)$. 
In the case when $\g S \geq 3$, since there is only one pair of pants in $S \setminus g(C\cup L^+\cup L^-)$ whose frontier contains $g(L^+) \cup g(L^-)$, this pair of pants must contain $l'$, and hence is homotopic relative to compact leaves to $g(\ell)$.

In the case when $\g S =2$, it is possible that $l'$ is contained in the pair of pants lying on the opposite side of $g(L^+ \cup L^- \cup c)$ from the one containing $g(\ell)$.
In this case $\iota(l')$ is homotopic relative to compact leaves to $g(\ell)$.
\end{proof}

Next, we shall take into account the direction in which a non-compact isolated leaf spirals around simple closed geodesics.
\begin{definition}
\label{unapproachable}
We call a non-compact isolated leaf $\ell$ of a finite lamination {\em incoherent} when it has only one limit component and if it spirals around this component on its two sides in the same direction. Otherwise, $\ell$ is called {\em coherent}.
\end{definition}

\begin{lemma}
\label{extension}
Let $\ell$ be a coherent non-compact isolated leaf of a finite geodesic lamination $\lambda$.
Then there is an approximable finite lamination $\lambda'$ containing $\ell$.
\end{lemma}
\begin{proof}
Let $L^+, L^-$ be the limit components of $\ell$ ($L^+$ and $L^-$ may coincide).
We can regard $\ell$ as obtained from an arc $a$ with endpoints lying on $L^+ \cup L^-$ by spiralling it around $L^+$ and $L^-$ infinitely many times.
We extend $a$ to a simple closed curve $c$ so that the endpoints of $a$ are essential intersection of $c$ with $L^+\cup L^-$, without adding an arc parallel to $a$.
By performing Dehn twists around $L^+$ and $L^-$ on $c$ infinitely many times in the same direction as the spiralling of $\ell$ and taking the Hausdorff limit, we get an approximable lamination as we wanted.
(Since $\ell$ is coherent, we can realise $\ell$ by an infinite iteration of Dehn twists.)
\end{proof}

In the case when $\g S =2$ we shall need another lemma.

\begin{lemma}
\label{genus 2}
Suppose that $\g S =2$, and let $\ell$ and $\ell'$ be two coherent non-compact isolated leaves of a finite lamination  which  have the following properties:
\begin{enumerate}[(a)]
\item $\ell$ has two distinct limit components $L^+$ and $L^-$.
\item
One of the limit components $L^+$ of $\ell$ is also a limit component of $\ell'$ whereas the other one, $L^-$, is not.
\item $\ell'$ has either one or two limit components. If $\ell'$ has only one limit component, then its ends spiral around the limit component $L^+$ on the same side of $L^+$.
\item
The leaves $\ell$ and $\ell'$ spiral around $L^+$ on the same side of $L^+$.
\end{enumerate}

Let $\mu_\ell$ denote the union $\ell \cup L^+ \cup L^-$, and $\mu_{\ell'}$ the union of $\ell'$ and its (one or two) limit components.
Then, there is an approximable geodesic lamination $\lambda'$ containing $\mu_\ell \cup \mu_{\ell'}$ and having a leaf which intersects both $\iota(\ell)$ and $\iota(\ell')$ transversely, where $\iota$ denotes as before the hyperelliptic involution.
\end{lemma}
\begin{proof}
If $\ell'$ has two limit components, let $L'$ be its limit component other than $L^+$.
If $\ell'$ has only one limit component, choose a closed geodesic disjoint from $L^+ \cup \ell \cup L^- \cup \ell'$, and let it be $L'$.
(By the property (c), such a closed geodesic exists.)
Then $L^+ \cup L^- \cup L'$ decompose $S$ into two pairs of pants, $P$ and $P'$.
By the properties (c) and (d), $\ell$ and $\ell'$ are contained in the same pair of pants, say $P$.
Now we can extend $\mu_\ell \cup \mu_{\ell'}$ to a geodesic lamination as we wanted by adding a leaf in $P'$ which intersects $\iota(\ell), \iota(\ell')$ transversely choosing the spiralling directions appropriately. 
\end{proof}

\begin{proposition}
\label{approachable}
Let $\ell$ be a coherent non-compact isolated leaf whose limit components $L^+, L^-$ are simple closed geodesics ($L^+$ and $L^-$ may coincide).
Let $\mu_\ell$ be the geodesic lamination $L^+ \cup \ell \cup L^-$.
Then $f(\mu_\ell)=g(\mu_\ell)$ when $\g S\geq 3$.
In the case when $\g S =2$, we have either $f(\mu_\ell)=g(\mu_\ell)$ or $f(\mu_\ell)=\iota \circ g(\mu_\ell)$, where $\iota$ denotes as before the hyperelliptic involution, and the alternative does not depend on $\ell$.
\end{proposition}
\begin{proof}
By \cref{extension}, there is an approximable finite lamination $\lambda$ containing $\mu_\ell$.
By \cref{approximable rigidity}, we have $f(\lambda)=g(\lambda)$.
On the other hand, if $\g S \geq 3$,  \cref{homotopy class} shows that $f(\mu_\ell)$ consists of $g(L^+) \cup g(L^-)$ together with a non-compact isolated leaf homotopic relative to compact leaves to $g(\ell)$.
Since $f(\lambda)=g(\lambda)$ contains $g(\ell)$, it cannot  contain leaves homotopic relative to compact leaves to $g(\ell)$ other than $g(\ell)$ itself.
Since $f(\mu_\ell)$ is contained in $f(\lambda)=g(\lambda)$,  
the only isolated leaf of $f(\mu_\ell)$ must coincide with $g(\ell)$.
Thus we have completed the proof in the case when $\g S\geq 3$. 

Suppose that $\g S=2$.
Then the same argument as in the case of $\g S \geq 3$ implies that $f(\mu_\ell)$ is either $g(\mu_\ell)$ or $\iota \circ g(\mu_\ell)$.
We need to show that one of the alternatives holds for all $\mu_\ell$.
First consider two non-compact isolated leaves $\ell$ and $\ell'$ as in the statement of \cref{genus 2}, and consider the approximating lamination $\lambda'$ provided by that statement.
Since $\lambda'$  has a leaf  $\ell''$ intersecting $\iota(\ell), \iota(\ell')$ transversely, if $f(\mu_\ell)=g(\mu_\ell)$, we cannot have $f(\mu_{\ell''})=\iota \circ g(\mu_{\ell''})$, for both $f(\mu_\ell)$ and $f(\mu_{\ell''})$ are contained in $f(\lambda')$, and hence we have $f(\mu_{\ell''})=g(\mu_{\ell''})$, and by the same argument $f(\mu_{\ell'})=g(\mu_{\ell'})$ holds.
In the same way, we see that  $f(\mu_\ell)=\iota \circ g(\mu_\ell)$ implies $f(\mu_{\ell'})=\iota \circ g(\mu_{\ell'})$.
Thus one of the alternatives holds for both $\mu_\ell$ and $\mu_{\ell'}$.

From now on, we shall only show that $f(\mu_\ell)=g(\mu_\ell)$ implies $f(\mu_{\ell'})=g(\mu_{\ell'})$ in a more general setting for $\ell$ and $\ell'$.
We shall omit the argument for the case where $f(\mu_\ell)=\iota\circ g(\mu_\ell)$, for every step goes in a parallel way.

If $\ell$ has only one limit component and spirals around it on its both sides, we have $\iota(\mu_\ell)=\mu_\ell$, since $\ell$ is coherent.
Therefore both alternatives hold for  such a case, and this can be excluded from the argument.
Except for this case, if two disjoint coherent non-compact isolated leaves $\ell$ and $\ell'$ lie in the same pair of pants $P$, then they must satisfy the conditions of \cref{genus 2}.
Therefore for two disjoint isolated non-compact leaves $\ell$ and $\ell'$, we see that  $f(\mu_\ell)=g(\mu_\ell)$ implies $f(\mu_{\ell'})=g(\mu_{\ell'})$.
If two coherent non-compact isolated leaves $\ell$ and $\ell'$ lying in the same pair of pants $P$ intersect, then we can choose either another non-compact isolated leaf $\ell''$ which is disjoint from both $\ell$ and $\ell'$ or a pair of disjoint non-compact isolated leaves $\ell''_1, \ell''_2$ such that $\ell\cap \ell''_1 = \emptyset, \ell''_2 \cap \ell'=\emptyset$, and therefore we get the same conclusion for $\ell$ and $\ell'$ as before.

Since $f$ is injective on $\gl(S)$, if $f(\mu_\ell)=g(\mu_\ell)$, then we must have $f(\iota(\mu_\ell))=g(\iota(\mu_\ell))$.
Combining this with what we have just proved, we see that if $\ell$ and $\ell'$ are two coherent non-compact isolated leaves whose limit components are simple closed geodesics and if there is a pants decomposition $P$ which is disjoint from both $\ell$ and $\ell'$, then $f(\mu_\ell)=g(\mu_\ell)$ implies $f(\mu_{\ell'})=g(\mu_{\ell'})$. 
We note that in particular, if $\ell$ and $\ell'$ are disjoint and their limit components coincide, then we can find a such a pants decomposition $P$.

Next suppose that $\ell$ and $\ell'$ have the same limit components $L^+, L^-$, but that they intersect each other.
Then we can find a sequence of coherent non-compact isolated leaves $\ell=\ell_1, \dots, \ell_k=\ell'$ having the same limit components $L^+, L^-$ such that $\ell_j \cap \ell_{j+1}=\emptyset$.
Therefore, by applying the above argument to $\ell_j$ and $\ell_{j+1}$ inductively, we see that $f(\mu_\ell)=g(\mu_\ell)$ implies  $f(\mu_{\ell'})=g(\mu_{\ell'})$.

Now suppose that $\ell$ and $\bar \ell$ are non-compact isolated leaves of possibly different finite laminations, both of which have two distinct limit components.
A procedure to replace one component $C$ of a pants decomposition $P$ by another curve $C'$ disjoint from $P \setminus C$ and to get a new pants decomposition $(P \setminus C) \cup C'$ is called an elementary move.
It is known that any two pants decompositions of $S$ can be joined by a composition of finitely many elementary moves (see for instance Hatcher \cite{Ha}).
It follows from this that there is a sequence of coherent non-compact isolated leaves $\ell=\ell_1, \ell_2, \dots, \ell_k=\bar \ell$  of finite laminations satisfying the following:
\begin{enumerate}[(i)]
\item $\ell_j$ is contained in a pair of pants of a pants decomposition $C_j$.
\item Either $C_{j+1}=C_j$ or $C_{j+1}$ is obtained from $C_j$ by an elementary move.
\item If $C_{j+!}\neq C_j$, then the limit components of $\ell_{j+1}$ coincide with those of $\ell_{j+1}$.
\end{enumerate}
Therefore, by what we have just proved up to the previous paragraph, we see that  $f(\mu_\ell)=g(\mu_\ell)$ implies $f(\mu_{\bar \ell})=g(\mu_{\bar \ell})$.
Thus we have completed the proof.
\end{proof}

\begin{corollary}
\label{all finite}
If $\g S \geq 3$, then for any finite geodesic lamination $\lambda$ that does not contain any incoherent non-compact isolated leaf, we have $f(\lambda)=g(\lambda)$.
If $\g S=2$, then $f(\lambda)=g(\lambda)$ for any such $\lambda$  or $f(\lambda)=\iota \circ g(\lambda)$ for any such $\lambda$.
\end{corollary}
\begin{proof}
By \cref{multicurves,approximable rigidity}, $f$ and $g$ coincide on the minimal components of $\lambda$, and by \cref{inclusion}, the number of the non-compact isolated leaves of $f(\lambda)$ is the same as that of $\lambda$.
Let $\ell$ be a non-compact isolated leaf of $\lambda$, which is coherent by assumption.
By \cref{approachable}, we have $f(\mu_\ell)=g(\mu_\ell)$ (or $f(\mu_\ell)=\iota \circ g(\mu_\ell)$ when $\g S=2$), and since $f(\lambda)$ contains $f(\mu_\ell)$, it must have $g(\ell)$ (or $\iota \circ g(\ell)$ when $\g S=2$) as a non-compact isolated leaf.
Since this holds for every non-compact isolated leaf, $f(\lambda)$ contains all non-compact isolated leaves of $g(\lambda)$ (or $\iota \circ g(\lambda)$ when $\g S=2$).
Since $f(\lambda)$ and $g(\lambda)$ have the same number of such leaves, which is equal to the number of non-compact isolated leaves of $\lambda$, we have $f(\lambda)=g(\lambda)$  (or $f(\lambda)=\iota \circ g(\lambda)$ when $\g S=2$).

In the case when $\g S=2$, by \cref{approachable} either $f(\mu_\ell)=g(\mu_\ell)$ for all $\lambda$ and  $\ell$ or $f(\mu_\ell)=\iota \circ g(\mu_\ell)$ for all $\lambda$ and $\ell$.
This shows the second part of our corollary.
\end{proof}

Now we turn to incoherent non-compact isolated leaves.

\begin{lemma}
\label{even unapproachable}
Let $\ell$ be an incoherent non-compact isolated leaf of a finite geodesic lamination, and $L$ its (unique) limit component.
Then for  $\mu_\ell=L \cup \ell$, we have $f(\mu_\ell)=g(\mu_\ell)$ if $\g S \geq 3$.
In the case when $\g S  =2$, we have either $f(\mu_\ell)=g(\mu_\ell)$ or $f(\mu_\ell)=\iota \circ g(\mu_\ell)$, and the alternative does not depend on $\ell$, nor on whether $\ell$ is incoherent or coherent.
\end{lemma}
\begin{proof}
Take a  simple closed geodesic $d$ in $S \setminus \mu_\ell$, and two coherent non-compact isolated leaves $\ell_1$ and $ \ell_2$ as follows:
\begin{enumerate}[1]
\item  $\ell_1$ and $\ell_2$ are disjoint, and are contained in $S \setminus (\mu_\ell \cup d)$.
\item For $j=1,2$, the ends of $\ell_j$ spiral around $d$ and $L$.
\item $\ell_1$ and $\ell_2$ spiral around $L$ on  opposite sides of $L$.
\end{enumerate}
Set $\nu_\ell$ to be $\mu_\ell \cup d \cup \ell_1\cup \ell_2$, and $\nu'_\ell$ to be $d\cup L \cup \ell_1 \cup \ell_2$.

Suppose that $\g S \geq 3$ for the moment.
By \cref{all finite}, we have $f(\nu'_\ell)=g(\nu'_\ell)$.
By \cref{homotopy class}, $f(\nu_\ell)$ has a non-compact isolated leaf $\ell'$ homotopic relative to compact leaves to $g(\ell)$.
Since $\ell'$ is disjoint from $f(\nu'_\ell)$, which must be contained in $f(\nu_\ell)$,  the direction of spiralling is the same as $g(\ell)$ at both ends, and hence $\ell'=g(\ell)$.
Thus we have $f(\nu_\ell)=g(\nu_\ell)$.

Next suppose that $\g S=2$.
By the same argument as in the case of $\g S \geq 3$, if $f(\nu'_\ell)=g(\nu'_\ell)$, we have $f(\nu_\ell)=g(\nu_\ell)$.
Otherwise, we have $f(\nu_\ell)=\iota \circ g(\nu_\ell)$.
Since one of the alternative holds for all $\nu'_\ell$ by \cref{all finite}, we see that the alternative does not depend on $\ell$.
\end{proof}

Now we can prove the following.
\begin{proposition}
\label{finite rigidity}
 If $\g S \geq 3$, we have $f(\lambda)=g(\lambda)$ for all finite geodesic laminations. If $\g S=2$, then $f(\lambda)=g(\lambda)$ for all finite geodesic laminations or $f(\lambda)=\iota \circ g(\lambda)$ for all finite geodesic laminations.
\end{proposition}
\begin{proof}
We first assume that $\g S \geq 3$.
Let $\lambda'$ be the union of the minimal components  and the coherent non-compact isolated leaves of $\lambda$.
By \cref{all finite}, $f(\lambda')=g(\lambda')$, and hence $f(\lambda)$ contains $g(\lambda')$.
Now, let $\ell$ be an incoherent non-compact isolated leaf of $\lambda$.
By \cref{even unapproachable}, $f(\lambda)$, which contains $f(\mu_\ell)=g(\mu_\ell)$, must contain $g(\ell)$.
Since this holds for every incoherent non-compact isolated leaf of $\lambda$,  $f(\lambda)$ contains $g(\lambda)$.
Since $f$ preserves the inclusions, the number of the leaves of $f(\lambda)$ is the same as that of $\lambda$, hence as that of $g(\lambda)$.
Therefore, the only possibility is $f(\lambda)=g(\lambda)$.

Now we turn to the case when $\g S=2$.
In this case, we have $f(\lambda')=g(\lambda')$ or $f(\lambda')=\iota \circ g(\lambda')$.
If the first possibility holds, this must hold for all $\lambda'$, and also we have $f(\mu_\ell)=g(\mu_\ell)$.
Therefore $f(\lambda)=g(\lambda)$ for every finite geodesic lamination $\lambda$.
Similarly, if $f(\lambda')=\iota \circ g(\lambda')$, then this holds for all $\lambda'$, and hence $f(\lambda)=\iota \circ g(\lambda)$ for every finite geodesic lamination $\lambda$.
\end{proof}

\section{Proof of the main theorem}
Now we are ready to prove \cref{main}.
\begin{proof}[Proof of \cref{main}]

We first show that if $f: \gl(S) \to \gl(S)$ is a bijection preserving left Hausdorff convergence, then there is an extended mapping class $h$ inducing the same bijection on $\gl(S)$.

Let $\lambda$ be a geodesic lamination.
Since finite laminations are dense in $\gl(S)$ with the Hausdorff topology (see \cref{finite laminations dense}), there is a sequence of finite laminations $\{\mu_i\}$ converging to $\lambda$.
By \cref{finite rigidity}, we have $f(\mu_i)=h(\mu_i)$ for some homeomorphism $h: S \to S$.
(This is either $g$ or $\iota \circ g$ in \cref{finite rigidity}.)
Since $h$ is a homeomorphism, $h(\lambda)$ coincides with the Hausdorff limit  $\mu_\infty$ of $h(\mu_i)=f(\mu_i)$.
Since $f$ preserves left Hausdorff convergence, $f(\lambda)$ contains the Hausdorff limit $\mu_\infty$.
As was seen before, $f$ preserves  the number of minimal components and the number of non-compact isolated leaves.
Thus, the only possibility is $f(\lambda)=\mu_\infty$, which is equal to $h(\lambda)$.

Thus, the natural homomorphism $\mathrm{Mod}^*(S)\to\mathrm{Aut}(\mathcal{GL}(S))$ is surjective. 

For $\g S\geq 3$, this homomorphism is injective since if two extended mapping classes induce the same bijection on $\mathcal{GL}(S)$, they induce the same action on the curve complex $C(S)$, and we know by Ivanov's result \cite{Iv} that the natural homomorphism $\mathrm{Mod}^*(S)\to \mathcal{C}(S)$ is injective.

 It remains to consider the case when $\g S=2$. We know that in this case,  if a homeomorphism $h$ of $S$ induces the identity map on the curve complex $\mathcal{C}(S)$, then $h$ is either homotopic to the identity or to the hyperelliptic involution $\iota$ of $S$. But the hyperelliptic involution does not induce the identity map on $\mathcal{GL}(S)$. To see this, take a geodesic pair of pants decomposition of $S$ which is invariant by $\iota$ up to homotopy, and complete it to a geodesic lamination by adding leaves which spiral along the three pants curves in a way that is not invariant by the hyperelliptic  involution $\iota$. Thus, $\iota$ does not induce the identity map on $\gl(S)$. This completes the proof. 
\end{proof}

Let us note finally that introducing the asymmetric Hausdorff distance on the space $\mathcal{GL}(S)$  opens up the way to a collection of questions in this new asymmetric setting. We mention for instance the study of the geodesics of this space (i.e.\ to describe the set of geodesics between any two points, to study their uniqueness, etc.), the study of its the boundary structure, and the relation between this distance function with the other distance functions and topologies on this space.

\medskip

\noindent {\bf Acknowledgement.} The authors are grateful to the anonymous referee who helped them to improve the writing in a substantial way.


\begin{thebibliography}{99}
  
   \bibitem{AMO} V. Alberge, H. Miyachi and K. Ohshika,  Null-set compactifications of Teichm\"uller spaces. Handbook of Teichm\"uller theory (ed. A. Papadopoulos). Vol. VI, 71-94, Eur. Math. Soc., Zürich, 2016.
  
  \bibitem{AB} Y. Algom-Kfir,  and M. Bestvina, Asymmetry of outer space. Geom. Dedicata 156 (2012), 81--92.
 

  
  \bibitem{CEG} R. Canary, D. B. A.  Epstein and P. Green, Notes on notes of Thurston. Analytical and geometric aspects of hyperbolic space (Coventry/Durham, 1984), 3--92, London Math. Soc. Lecture Note Ser., 111, Cambridge Univ. Press, Cambridge, 1987.
  
   \bibitem{BC} A.  Casson and S. Bleiler, Automorphisms of surfaces after Nielsen and Thurston. 
London Mathematical Society Student Texts, 9. Cambridge University Press, Cambridge, 1988.  
  \bibitem{CPP} C. Charitos,  I. Papadoperakis and  A. Papadopoulos,  On the homeomorphisms of the space of geodesic laminations on a hyperbolic surface. Proc. Amer. Math. Soc. 142 (2014),  2179--2191.
  
  \bibitem{Disarlo} V. Disarlo, Combinatorial rigidity of arc complexes, arXiv:1505.08080 (2015).


\bibitem{DGK}  J. Danciger, F. Gu\'{e}ritaud and F. Kassel, Geometry and topology of complete Lorentz spacetimes of constant curvature. Ann. Sci. \'Ec. Norm. Sup\'{e}r. (4) 49 (2016), no. 1, 1--56.
 
\bibitem{GLM}  W. M. Goldman,  F. Labourie and G.  Margulis,  Proper affine actions and geodesic flows of hyperbolic surfaces. Ann. of Math. (2) 170 (2009),  1051-1083. 

\bibitem{GK} F. Gu\'{e}ritaud and F. Kassel, 
Maximally stretched laminations on geometrically finite hyperbolic manifolds. 
Geom. Topol. 21 (2017),  693--840. 

  \bibitem{Ha} A. Hatcher, On triangulations of surfaces. Topology and Its Applications, 40 (2)  (1991), 189--194. 
  
  
\bibitem{IM} E. Irmak and J. D. McCarthy, Injective simplicial maps of the arc complex, Injective Simplicial Maps of the Arc Complex,  Turkish J. of Math., 33 (2009) 1--16.

  \bibitem{Iv} N. Ivanov, Automorphism of complexes of curves and of Teichm\"{u}ller spaces.  
Internat. Math. Res. Notices, 1997, no. 14, 651--666. 




\bibitem{KP1} M. Korkmaz and A. Papadopoulos, On the arc and curve complex of a surface, Math. Proc. Camb. Philos. Soc. 148, 473-483 (2010).



\bibitem{2012b} L. Liu,  A. Papadopoulos, W. Su and G. Th\'eret,  On the classification of mapping class actions on Thurston's asymmetric metric, Math. Proc. of the Cambridge Philosophical Society, Volume  155 (2013),  499--515.
 

\bibitem{Markovic2003} V. Markovic, Biholomorphic maps between Teichm\"uller spaces.  Duke Mathematical Journal, 120  (2003),  405--431.


\bibitem{MW2002} H. Masur and M. Wolf, The Weil-Petersson isometry group, \textit{Geom. Dedicata} 93 (2002), 177--190.


 \bibitem{MT}  K.\ Matsuzaki and M.\ Taniguchi, Hyperbolic manifolds and Kleinian groups. Oxford Mathematical Monographs. Oxford Science Publications. The Clarendon Press, Oxford University Press, New York, 1998. 
 


\bibitem{MCP} J. D. McCarthy and A. Papadopoulos, Simplicial actions of mapping class groups. In: Handbook of Teichm\"uller theory, Vol. III  (ed. A. Papadopoulos), European Mathematical Society, Z\"urich, 2012,   297--423.


\bibitem{SM} S. Meinert, The Lipschitz metric on deformation spaces of $G$-trees, Algebraic \& Geometric topology 15 (2015), 987--1029.



  

\bibitem{Oh} K. Ohshika, A note on the rigidity of unmeasured lamination spaces. Proc. Amer. Math. Soc. 141 (2013), 4385--4389.

\bibitem{Oh1}  K. Ohshika,  Reduced Bers boundaries of Teichm\"{u}ller spaces.  Ann. Inst. Fourier (Grenoble) 64 (2014),  145--176.
 
\bibitem{OP} K. Ohshika and A. Papadopoulos, Hom\'{e}omorphismes et nombre d'intersection, Comptes Rendus Acad. Sciences Paris, Math\'{e}mathiques,  Ser. I, 356 (2018), 899--902.

\bibitem{P1991} A. Papadopoulos, On Thurston's boundary of Teichm\"uller space and the extension of
earthquakes. Topology Appl. 41 (1991), 147--177.

 \bibitem{2007a} A. Papadopoulos and  G. Th\'eret, 
On the topology defined by Thurston's asymmetric metric,
Math. Proc. Camb. Philos. Soc. 142 (2007),  487--496. 


   \bibitem{Papadopoulos2008} A. Papadopoulos,  Measured foliations and mapping class groups of surfaces. Balkan J. Geom. Appl. 13 (2008), 93--106 .

\bibitem{Pa} A. Papadopoulos, A rigidity theorem for the mapping class group action on the space of unmeasured foliations on a surface, Proc. Amer. Math. Soc. 136 (2008), 4453--4460


\bibitem{2009j}  A. Papadopoulos and G. Th\'eret, Some Lipschitz maps between hyperbolic surfaces with applications to Teichm\"uller theory,   Geometriae Dedicata, 150 (2011),  233--247. 


\bibitem{Pa1} A.   Papadopoulos, Actions of mapping class groups. Handbook of group actions. Vol. I (ed. L. Ji, A. Papadopoulos, S.-T. Yau), Adv. Lect. Math. (ALM), 31, Int. Press, Somerville, MA, and Higher Education Press, Beijing, 2015,  189--248.


\bibitem{Royden1971} H. L. Royden, Automorphisms and isometries of Teichm\"uller space, in: \textit{Advances in the theory of Riemann surfaces}, Ann. of Math. Studies, vol. 66 (1971), 317-328.


 
 

\bibitem{T32} O.  Teichm\"uller, Ver\"anderliche Riemannsche Fl\"achen. Deutsche Math. 7, 344-359 (1944).
 English translation  by A. A'Campo Neuen, Variable Riemann surfaces, In: 
 Handbook of Teichm\"uller theory, Vol. IV  (ed. A. Papadopoulos), European Mathematical Society, Z\"urich,
 2014,  787--803.

 \bibitem{Theret} G. Th\'eret, On elementary antistretch lines. Geom. Dedicata 136 (2008), 79--93.


   \bibitem{Th} W. P. Thurston, The geometry and topology of three-manifolds, Princeton University Lecture Notes, 1979.

   
   \bibitem{Th1} W. P. Thurston, Minimal stretch maps between hyperbolic surfaces, preprint, 1986, arXiv:math/9801039.
   
   

\bibitem{Walsh-Hand} C. Walsh, The horoboundary and isometry group of Thurston's Lipschitz metric, In: Handbook of Teichm\"uller theory, Vol. IV  (ed. A. Papadopoulos), European Mathematical Society, Z\"urich,
2014, 327--353. 

      \end{thebibliography}
\end{document}